\documentclass[10pt, english, draft]{article}
\usepackage{amssymb,amsmath}
\title{\bf Some Infinite Matrices
Whose Leading Principal Minors Are Well-known Sequences}
\author{ {\sc A. R. Moghaddamfar}$^1$, {\sc S. Rahbariyan}$^1$,\\
 {\sc S. Navid Salehy}$^2$  and  {\sc S. Nima Salehy}$^2$\\[0.3cm]
{\em $^1$Department of Mathematics,}\\
{\em K. N. Toosi University of Technology,}\\[0.1cm]
{\em P. O. Box
$16315$-$1618$, Tehran, Iran}\\[0.1cm]
{\em $^2$Department of Mathematics, Florida State
University,}\\{\em
Tallahassee, FL $32306$, USA.}\\[0.1cm]
{\em E-mails}: {\tt moghadam@ipm.ir} \ and \ {\tt
moghadam@kntu.ac.ir}\\[0.2cm]{\em In memory of Professor Michael
Neumann}}

\newenvironment{proof}{\noindent {\em {Proof}}.}{$\square$
\medskip}

\newtheorem{definition}{Definition}
\newtheorem{corollary}{Corollary}

\newtheorem{theorem}{Theorem}

\newtheorem{lemma}{Lemma}

\begin{document}
\maketitle
\begin{abstract}
\noindent There are scattered results in the literature showing
that the leading principal minors of certain infinite integer
matrices form the Fibonacci and Lucas sequences. In this article,
among other results, we have obtained new families of infinite
matrices such that the leading principal minors of them form a
famous integer (sub)sequence, such as Fibonacci, Lucas, Pell and
Jacobsthal (sub)sequences.
\end{abstract}

\renewcommand{\baselinestretch}{1.1}
\def\thefootnote{ \ }
\footnotetext{{\em $2000$ Mathematics Subject Classification}:
15A15, 15A23, 15B05, 15B36, 11C20.}

\section{Notation, Definitions and Previous Results}
\indent Throughout this article, unless noted otherwise, we will
use the following notation:
\begin{itemize}
\item $\alpha=(\alpha_i)_{i\geqslant 0}$ and $\beta=(\beta_i)_{i\geqslant 0}$ are
two arbitrary sequences starting with $\alpha_0=\beta_0=\gamma$.

\item If $\alpha=(\alpha_i)_{i\geqslant 0}$, then $\tilde{\alpha}=(\tilde{\alpha}_i)_{i\geqslant 0}$
where $\tilde{\alpha}_i=(-1)^i\alpha_i$ for all $i\geqslant 0$.

\item  $P_{\alpha, \beta}(n)$ is the generalized Pascal
triangle associated with the sequences $\alpha$ and $\beta$ (see
\cite{bacher}), which is introduced as follows. In fact,
$P_{\alpha, \beta}(n)=[P_{i,j}]_{0\leqslant i, j\leqslant n}$ is a
square matrix of order $n+1$ whose $(i, j)$-entry $P_{i,j}$ obeys
the following rule:
$$P_{i,0}=\alpha_i, \  P_{0,i}=\beta_i  \ \ \  \mbox{for} \ 0\leqslant i\leqslant n,$$
and $$P_{i,j}=P_{i,j-1}+P_{i-1,j} \ \  \mbox{for} \ \ 1\leqslant
i, j\leqslant n.$$
\item $A_{\alpha, \beta}(n)=[A_{i,j}]_{0\leqslant i, j\leqslant n}$ is
the 7-matrix associated with the sequences $\alpha$ and $\beta$
of order $n+1$, whose entries satisfy
$$ A_{i,0}=\alpha_i, \
A_{0,i}=\beta_i  \ \ \  \mbox{for} \  0\leqslant i\leqslant n,$$
and $$A_{i,j}=A_{i-1,j-1}+A_{i-1,j} \ \ \mbox{for} \ \ 1\leqslant
i, j\leqslant n.$$ This class of matrices was first introduced in
\cite{chrs}. In applications, the case when $\alpha=(1, 1, 1,
\ldots)$ and $\beta=(1, 0, 0, 0, \ldots)$ is more important. In
fact, we put
 $$L(n):=A_{(1, 1, 1,
\ldots), (1, 0, 0, 0, \ldots)}(n)=[L_{i,j}]_{0\leqslant i,
j\leqslant n},$$ which is called the unipotent lower triangular
matrix of order $n+1$. An explicit formula for $(i,j)$-entry
$L_{i,j}$ is also given by the following formula
$$L_{i,j}=\left \{ \begin{array}{ccc}
0 & \mbox{if} & i<j\\[0.1cm]
{i\choose j} & \mbox{if} & i\geqslant j.
\end{array} \right.
$$ Moreover, we put $U(n)=L(n)^t$, where $A^t$ signifies the transpose of matrix
$A$.

\item  $T_{\alpha, \beta}(n)=[T_{i,j}]_{0\leqslant i, j\leqslant n}$ is
the Toeplitz matrix of order $n+1$ with $T_{i,0}=\alpha_i$ and
$T_{0, i}=\beta_i$, $0\leqslant i\leqslant n$, and $T_{i,
j}=T_{k, l}$ whenever $i-j=k-l$. In the case when $\alpha=(c, b,
b, b, \ldots)$ and $\beta=(c, a, a, a, \ldots)$, we follow the
notation in \cite{GST}, and put $T_{n+1}(a, b, c):=T_{\alpha,
\beta}(n)$. For example
$$T_3(a, b, c)=\left[\begin{array}{lll} c & a & a\\ b & c & a\\b & b & c\\\end{array}\right].$$
\item  An upper Hessenberg matrix, $H(n)=[h_{i,j}]_{0\leqslant i,j\leqslant n}$,
is a square matrix of order $n+1$ where $h_{i,j}=0$ whenever
$i>j+1$ and $h_{i,i-1}\neq 0$ for some $i$, $1\leqslant
i\leqslant n$, so we have
$$H(n)=\left ( \begin{tabular}{cccccc}
$h_{0,0}$ & $h_{0,1}$ & $h_{0,2}$ & \ \ $\ldots$ &  $h_{0,n}$  \\
$h_{1,0}$ & $h_{1,1}$ & $h_{1,2}$ & \ \ $\ddots$ & $h_{1,n}$  \\
$0$ & $h_{2,1}$ & $h_{2,2}$ & \ \ $\ddots$ & $h_{2,n}$  \\
$\ \ \vdots$ & $\ \ \ddots$ & \ \ $\ddots$ & \ \ $\ddots$ & $h_{n-1,n}$  \\
$0$ & $0$ & $0$ & $h_{n,n-1}$ & $h_{n,n}$  \\
\end{tabular} \right ).$$
\item $E_{ij}$ denotes the square matrix having 1 in the $(i,
j)$ position and 0 elsewhere.
\item  ${\rm R}_i(A)$ (resp. ${\rm C}_j(A)$) denotes the row
$i$ (resp. column $j$) of a matrix $A$.
\end{itemize}
{\sc Remark.} {\em Notice that we index all matrices in this
article beginning at the $(0,0)$-th entry.}
\begin{itemize}
\item  $\mathcal{F}_n$  is the $n$th Fibonacci number (A000045 in \cite{integer}), which
satisfies
$$\mathcal{F}_0=0, \ \mathcal{F}_1=1, \ \ \ \  \mathcal{F}_{n}=\mathcal{F}_{n-1}+\mathcal{F}_{n-2} \ \ \mbox{for} \ n\geqslant 2. $$
\item   $\mathcal{L}_n$  is the $n$th Lucas number (A000032 in \cite{integer}),
which satisfies
$$\mathcal{L}_0=2, \ \mathcal{L}_1=1, \ \ \ \  \mathcal{L}_{n}=\mathcal{L}_{n-1}+\mathcal{L}_{n-2}  \ \ \mbox{for} \ n\geqslant 2.$$
\item   $\mathcal{P}_n$  is the $n$th Pell number (A000129 in \cite{integer}),
which satisfies
$$\mathcal{P}_0=0, \ \mathcal{P}_1=1, \ \ \ \  \mathcal{P}_{n}=2\mathcal{P}_{n-1}+\mathcal{P}_{n-2}  \ \ \mbox{for} \ n\geqslant 2.$$
\item   $\mathcal{J}_n$  is the $n$th Jacobsthal number
(A001045 in \cite{integer}), which satisfies
$$\mathcal{J}_0=0, \ \mathcal{J}_1=1, \ \ \ \  \mathcal{J}_{n}=\mathcal{J}_{n-1}+2\mathcal{J}_{n-2}  \ \ \mbox{for} \ n\geqslant 2.$$
\end{itemize}
{\sc Remark.} Let $a, b, r, s$ be integers with $r, s\geqslant
1$. The {\em $(a, b, r, s)$-Gibonacci sequence} (or {\em
generalized Fibonacci sequence}),
$\mathcal{G}^{(a,b,r,s)}=(\mathcal{G}^{(a,b,r,s)}_n)_{n\geqslant
0 }$, is recursively defined by:
$$ \mathcal{G}^{(a, b, r, s)}_0=a, \ \ \mathcal{G}^{(a, b, r, s)}_1=b,$$
and $$\mathcal{G}^{(a, b, r, s)}_n=r\mathcal{G}^{(a, b, r,
s)}_{n-1}+s\mathcal{G}^{(a, b, r, s)}_{n-2},
  \ \ \mbox{for} \ n\geqslant 2.$$
The Fibonacci sequence (resp. Lucas sequence, Pell sequence, or
Jacobsthal sequence) corresponds to the case $(a, b, r, s)=(0, 1,
1, 1)$ (resp. $(2, 1, 1, 1)$, $(0, 1, 2, 1)$ or $(0, 1, 1, 2)$).

\begin{itemize}
\item Given an arbitrary sequence $\omega=(\omega_i)_{i\geqslant
0}$, we put $\tilde{\omega}=(\tilde{\omega}_i)_{i\geqslant 0}$
where $\tilde{\omega}_i=(-1)^i\omega_i$.
\item To any sequence $\alpha=(\alpha_i)_{i\geqslant 0}$, we
associate other sequences
$\hat{\alpha}=(\hat{\alpha}_{i})_{i\geqslant 0}$ and
$\check{\alpha}=(\check{\alpha}_{i})_{i\geqslant 0}$ (the
binomial and inverse binomial transforms of $\alpha$), which are
defined through the following rules:
$$\hat{\alpha}_{i}=\sum_{k=0}^{i}(-1)^{i+k}{i\choose k}\alpha_k  \ \ \
\mbox{and} \ \ \ \check{\alpha}_{i}=\sum_{k=0}^{i}{i\choose
k}\alpha_k.$$
\item  $\phi=\frac{1+\sqrt{5}}{2}$ is the golden ratio and
$\Phi=\frac{1-\sqrt{5}}{2}$ is the golden ratio conjugate.
\end{itemize}

Given an infinite matrix $A=[A_{i,j}]_{i,j\geqslant 0}$, denote by
$d_n$ ($n= 0, 1, 2, \ldots$) the $n$th leading principal minor of
$A$ defined as the determinant of the submatrix consisting of the
entries in its first $n+1$ rows and columns, i.e., $d_n=\det
[A_{i,j}]_{0\leqslant i,j\leqslant n}$. Here, we are interested in
computing the sequence of leading principal minors $(d_0, d_1,
d_2, \ldots)$, especially, in the case that $d_n$ is the $n$th
Fibonacci (Lucas, Pell or Jacobsthal) number. In \cite{mp}, we
introduced a family of matrices $A(n)=[A_{i,j}]_{0\leqslant i,
j\leqslant n}$ as follows:
$$A_{i,j}=\left\{\begin{array}{lll} 1 &  & ij=0, \\[0.2cm]
A_{i,j-1}+A_{i-1,j}+i-j &  &  1\leqslant i, j\leqslant n.
\end{array} \right.$$
The matrix $A(3)$ for instance is given by
$$ A(3)=\left (
\begin{array}{cccc} 1 & 1 & 1 & 1  \\
1 & 2 & 2 & 1  \\
1 & 4 & 6 & 6  \\
1 & 7 & 14 & 20  \\
\end{array} \right ).
$$
The leading principal minors now generate the Fibonacci numbers
$\mathcal{F}_{n+1}$ (see \cite{mp}).

In \cite{armt}, we defined two other families of matrices:
$$B(n)=[B_{i,j}]_{0\leqslant i, j\leqslant n} \ \ \mbox{and} \ \ C(n)=[C_{i,j}]_{0\leqslant i,
j\leqslant n},$$ with
$$B_{i,j}=\left\{\begin{array}{lll} j+2 &  & i=0,  j\geqslant 0, \\[0.2cm]
4B_{i-1,j}+i^2-7i-5 &  & j=0,  i\geqslant 1, \\[0.2cm]
B_{i,j-1}+B_{i-1,j}-2(i+j)&  &  1\leqslant i, j\leqslant n,
\end{array} \right.$$
and
$$C_{i,j}=\left\{\begin{array}{lll}
2-j & & i=0, j\leqslant 1, \\[0.2cm]
2C_{i, j-2}-C_{i, j-1} &  & i=0,  j\geqslant 2, \\[0.2cm]
3C_{i-1, j}+5(3^{i-1}-2i-1)/2 &  & j=0,  i\geqslant 1, \\[0.2cm]
C_{i-1,j-1}+C_{i-1,j}-2i&  &  1\leqslant i, j\leqslant n.
\end{array} \right.$$
The matrices $B(3)$ and $C(3)$, for instance, are shown below:
$$B(3)=\left (
\begin{array}{cccc}
2& 3& 4& 5 \\ -3& -4& -6& -9 \\  -27& -37&
-51& -70 \\ -125& -170& -231& -313 \\
\end{array} \right ),$$ and $$ C(3)=\left (
\begin{array}{cccc}
2& 1& 3& -1 \\
1& 1& 2& 0\\
-2& -2&-1& -2 \\
-1& -10& -9& -9 \\
\end{array} \right ).
$$
It was proved in \cite{armt} that $\det B(n)=\det
C(n)=\mathcal{L}_n$.

Recently, in \cite{mtaj}, it is determined all sequences $\alpha$
and $\beta$ which satisfy the second-order homogeneous linear
recurrence relations and for which the leading principal minors
of generalized Pascal triangle $P_{\alpha, \beta}=P_{\alpha,
\beta}(\infty)$, 7-matrix $A_{\alpha, \beta}=A_{\alpha,
\beta}(\infty)$ or Toeplitz matrix $T_{\alpha, \beta}=T_{\alpha,
\beta}(\infty)$, form the Fibonacci, Lucas, Pell and Jacobsthal
sequences.

In \cite{yang}, using generating functions, the authors proved
that any sequence can be presented in terms of a sequence of
leading principal minors of an infinite matrix. In fact,
according to their methods one can easily construct matrices whose
determinants are famous sequences. For example, they constructed
the following matrix
$$D(\infty)=\left (
\begin{array}{ccccc}
1& 1& 1& 1 &   \ldots \\
4& 1& 2& 3 &  \ldots \\
9& 0& 1& 3 &   \ldots \\
16& 0& 0& 1 &  \ldots \\
\vdots& \vdots&    \vdots& \vdots& \ddots \\
\end{array} \right ),$$
in which the first column is the sequence of the square positive
integers and the remaining part of the matrix is the truncated
Pascal's upper triangular matrix. Let $D(n)$ denote the $n+1$ by
$n+1$ upper left corner matrix of $D(\infty)$. Then $\det
D(n)=(-1)^n \mathcal{T}_{n+1}$, where $\mathcal{T}_n$ is the $n$th
triangular number, i.e., $ \mathcal{T}_n={n+1\choose 2},
n\geqslant 0$.

In \cite{GST}, in particular, it was proved that, for Toeplitz
matrices $$P_n=T_n(\phi, \Phi, 1) \ \ \ \ \mbox{and} \ \ \ \
Q_n=T_n(-\phi, -\Phi, 0),$$  $\det P_n=\mathcal{F}_{n+1}$ and
$\det Q_n=\mathcal{F}_{n-1}$.

Table 1-3 list several other infinite matrices, whose leading
principal minors form a Fibonacci or Lucas (sub)sequence. In these
tables $\mathbf{i}$ denotes $\sqrt{-1}$.

\begin{center}
{\small {\rm \bf Table 1.}  {\em Some Toeplitz matrices with determinants as Fibonacci or Lucas numbers.}} \\[0.2cm]
$\begin{array}{|l|l|l|c|} \hline
  \alpha & \beta & \det T_{\alpha, \beta}(n) & {\rm Refs.} \\
\hline
 (2, 1, 1, \ldots) & (2, -1, 0, 0, \ldots) & \mathcal{F}_{2n+3} & \cite{cahill1} \\[0.1cm]
 (2, 1, 1, \ldots) & (2, 1, 0, 0, \ldots) & \mathcal{F}_{n+3} & \cite{cahill1} \\[0.1cm]
 (1, \mathbf{i}, 0, 0, \ldots) & (1, \mathbf{i}, 0, 0, \ldots)
& \mathcal{F}_{n+2}
& \cite{bens, cahill1, cahill2, yang} \\[0.1cm]
 (1, -1, 0, 0, \ldots) & (1, 1, 0, 0, \ldots) & \mathcal{F}_{n+2} & \cite{St1} \\[0.1cm]
 (3, 1, 0, 0, \ldots) & (3, 1, 0, 0, \ldots) & \mathcal{F}_{2n+4} & \cite{bens, St1, St2} \\[0.1cm]
 (3, -1, 0, 0, \ldots) & (3, -1, 0, 0, \ldots) & \mathcal{F}_{2n+4} & \cite{St2} \\[0.1cm]
 (2, -1, 1, -1, 1, \ldots) & (2, 1, 0, 0, \ldots) & \mathcal{F}_{2n+3} &  \cite{bens} \\[0.1cm]
 (2, -1, 1, -1, 1, \ldots) & (2, -1, 0, 0, \ldots) &\mathcal{F}_{n+3} &  \cite{bens} \\[0.1cm]
 (1, \Phi, \Phi, \ldots) & (1, \phi, \phi, \ldots) &
\mathcal{F}_{n+2}
& \cite{GST} \\[0.1cm]
 (0, -\Phi, -\Phi, \ldots) & (0, -\phi, -\phi, \ldots) &
\mathcal{F}_{n}
& \cite{GST} \\[0.1cm]
\hline
\end{array}
$\\[0.5cm]
{\small {\rm \bf Table 2.}  {\em Some  matrices with determinants as Fibonacci or Lucas numbers}}. \\[0.4cm]
$\begin{array}{|l|l|l|c|} \hline
 \alpha & \beta & {\rm Determinant} & {\rm Ref.} \\
\hline  (1, \mathbf{i}, 0, \ldots) & (1, \mathbf{i}, 0, \ldots) &
\det \big(T_{\alpha, \beta}(n)+E_{11}\big)=\mathcal{L}_{n+1}
& \cite{cahill1} \\[0.1cm]
 (1, 1, 1, \ldots) & (1, -1, 0,  \ldots) & \det
\big(T_{\alpha,
\beta}(n)+\sum_{i=1}^{n}E_{ii}\big)=\mathcal{F}_{2n+2}
& \cite{cahill1} \\[0.1cm]
 (1, \mathbf{i}, 0,  \ldots) & (1, \mathbf{i}, 0,  \ldots)
& \det \big(T_{\alpha, \beta}(n)+2E_{00}\big)=\mathcal{L}_{n+2}
& \cite{by} \\[0.1cm]
\hline
\end{array}
$\\[0.5cm]
{\small {\rm \bf Table 3.}  {\em Some Pascal and 7-matrices with determinants as Fibonacci or Lucas numbers}}. \\[0.4cm]
$\begin{array}{|l|l|l|l|c|} \hline
 \gamma & \alpha & \beta & {\rm Determinant} & {\rm Ref.} \\
\hline 1 & \alpha_i=\alpha_{i-1}+c & \beta_i=\beta_{i-1}-c^{-1} &
\det
P_{\alpha, \beta}(n)=\mathcal{F}_{n+2}& \cite{mpss}\\[0.1cm]
1 & \alpha_i=\alpha_{i-1}+1 & \beta_1=0, \ \beta_i=\beta_{i+2} &
\det A_{\alpha, \beta}(n)=\mathcal{F}_{n+1} & \cite{mpss} \\[0.1cm]
1 & \alpha_i=\alpha_{i-1}-1 & \beta_1=0, \ \beta_i=\beta_{i+2} &
\det A_{\alpha, \beta}(n)=\mathcal{F}_{n+1} & \cite{mpss} \\[0.1cm]
1 & \alpha_i=\alpha_{i-1}+\mathbf{i} & \beta_1=2\mathbf{i}, \
\beta_i=\beta_{i+2} &
\det A_{\alpha, \beta}(n)=\mathcal{L}_{n+1} & \cite{mpss} \\[0.1cm]
1 & \alpha_i=\alpha_{i-1}-\mathbf{i} & \beta_1=-2\mathbf{i}, \
\beta_i=\beta_{i+2} &
\det A_{\alpha, \beta}(n)=\mathcal{L}_{n+1}  & \cite{mpss} \\[0.1cm]
\hline
\end{array}
$\\[0.5cm]
\end{center}

The purpose of this article is to find some {\em new infinite
families of matrices} whose leading principal minors form a
subsequence of the sequences $\mathcal{F}$, $\mathcal{L}$,
$\mathcal{P}$ and $\mathcal{J}$.
\section{Preliminaries}
In this section, we prove several auxiliary results to be used
later. First of all, we introduce the notion of equimodular
matrices (see \cite{zp}).

\begin{definition} Infinite matrices $A=[A_{i,j}]_{i,j\geqslant 0}$ and
$B=[B_{i,j}]_{i,j\geqslant 0}$ are equimodular if their leading
principal minors are the same, that is
$$\det [A_{i,j}]_{0\leqslant i,j\leqslant n}=\det [B_{i,j}]_{0\leqslant i,j\leqslant
n} \ \ \mbox{for all } \ n\geqslant 0.$$
\end{definition}

\begin{lemma}[Corollary 2.6, \cite{mtaj}]\label{cor-1}
Let $\alpha$ and $\beta$ be two sequences starting with a common
first term. Then, the following matrices are equimodular:
$$T_{\alpha, \beta}(\infty), \ \ T_{\beta, \alpha}(\infty), \ \ P_{\check{\alpha},
\check{\beta}}(\infty), \ \ P_{\check{\beta},
\check{\alpha}}(\infty),  \ \ A_{\check{\alpha}, {\beta}}(\infty)
\ \ {\rm and} \ \ A_{\check{\beta}, {\alpha}}(\infty).$$
\end{lemma}

The following lemma, which is taken from \cite{GST}, determines
recursive and explicit formula for the leading principal minors
for special classes of Toeplitz matrices.
\begin{lemma}\label{lem1} (\cite{GST}) Let $a, b, c\in \Bbb{C}$. For each positive integer $n$, let
$T_n=T_n(a, b, c)$. Then $\det T_1=c$, $\det T_2=c^2-ab$ and for
$n\geqslant 3$,
$$\det T_n=(2c-a-b)\det T_{n-1}-(c-a)(c-b)\det T_{n-2}.$$
Moreover, an explicit formula for the case when $n\geqslant 3$,
is as follows:
$$\det T_n=\left\{\begin{array}{lll}
0 & \mbox{if} & a=b=c,\\[0.3cm]
[c+a(n-1)](c-a)^{n-1} & \mbox{if} & a=b\neq c,\\[0.3cm]
\frac{b}{b-a}(c-a)^n-\frac{a}{b-a}(c-b)^n & \mbox{if} & a\neq b.\\[0.3cm]
\end{array}\right.$$
\end{lemma}

\begin{lemma}\label{lem2}  Let $a, c\in \Bbb{C}$. Let $\alpha=(\alpha_i)_{i\geqslant 0}$ with $\alpha_0=c$ and
for each $i\geqslant 1$, $\alpha_i=a$. Then there hold:

$(1)$  for each $i\geqslant 0$, $\check{\alpha}_i=c+a(2^i-1)$.

$(2)$ $\hat{\alpha}_0=c$ and for each $i\geqslant 1$,
$\hat{\alpha}_i=\left\{\begin{array}{ll}
a-c &  i \ \mbox{is odd,}\\[0.2cm]
c-a & i \  \mbox{is even.}\\ \end{array}\right.$
\end{lemma}
\begin{proof}
The following easy computations
$$\check{\alpha}_i=\sum_{k=0}^{i}{i\choose
k}\alpha_k=c+a\sum_{k=1}^{i}{i\choose k}=c+a(2^i-1),$$ and
$$\begin{array}{lll} \hat{\alpha}_i&=&\sum_{k=0}^{i}(-1)^{i+k}{i\choose
k}\alpha_k=(-1)^ic+a\sum_{k=1}^{i}(-1)^{i+k}{i\choose k}\\[0.5cm]
&=&(-1)^i(c-a)=\left\{\begin{array}{ll}
a-c &  i \ \mbox{is odd,}\\[0.2cm]
c-a & i \  \mbox{is even,}\\ \end{array}\right.\end{array}$$
conclude the results.
\end{proof}

Unless noted otherwise, we use the following notation in this
section:
\begin{itemize}
\item $\mathcal{G}=\mathcal{G}^{(a,b,r,1)}$, where $a, b$ and $r$ are
integers with $r\geqslant 1$,

\item $\bar{\mathcal{G}}=\mathcal{G}^{(0,1,r,1)}$, where $r\geqslant 1$ is
an integer. Moreover, we assume that $\bar{\mathcal{G}}_{-1}=1$.
\end{itemize}

\begin{lemma}\label{lem4}
The sequences
$\check{\mathcal{G}}=(\check{\mathcal{G}}_i)_{i\geqslant 0}$ and
$\check{\tilde{\mathcal{G}}}=(\check{\tilde{\mathcal{G}}}_i)_{i\geqslant
0}$ satisfy linear recursions of order $2$. More precisely, we
have
\begin{equation}\label{e1-new-2012}
\left\{\begin{array}{l} \check{\mathcal{G}}_0=\mathcal{G}_0=a,
 \ \ \ \check{\mathcal{G}}_1=\mathcal{G}_0+\mathcal{G}_1=a+b, \\[0.2cm]
\check{\mathcal{G}}_i=(r+2)\check{\mathcal{G}}_{i-1}-r\check{\mathcal{G}}_{i-2},
\ \ \ i\geqslant 2,
\end{array} \right.
\end{equation}
and
\begin{equation}\label{e2-new-2012}
\left\{\begin{array}{l} \check{\tilde{\mathcal{G}}}_0=\mathcal{G}_0=a, \ \ \ \check{\tilde{\mathcal{G}}}_1=\mathcal{G}_0-\mathcal{G}_1=a-b, \\[0.2cm]
\check{\tilde{\mathcal{G}}}_i=(2-r)\check{\tilde{\mathcal{G}}}_{i-1}+r\check{\tilde{\mathcal{G}}}_{i-2},
\ \ \ i\geqslant 2.
\end{array} \right.
\end{equation}
\end{lemma}
\begin{proof}
By the definition, we have the following recurrence relation
\begin{equation}\label{eee-qqq1}
\mathcal{G}_n=r\mathcal{G}_{n-1}+\mathcal{G}_{n-2}, \end{equation}
which can be written as
$$\mathcal{G}_n-r\mathcal{G}_{n-1}-\mathcal{G}_{n-2}=0.$$
Its characteristic equation is $$x^2-rx-1=0,$$ and we denote by
$t_1$ and $t_2$ the characteristic roots. Now, the general
solution of Eq. (\ref{eee-qqq1}) is given by
\begin{equation}\label{eee-qqq2}
\mathcal{G}_i=C_1t_1^i+C_2t_2^i,
\end{equation} where $C_1$ and $C_2$
are constants. In order to verify Eq. (\ref{e1-new-2012}), we
first notice that, for every $i\geqslant 0$:
$$
\begin{array}{lll}
\check{\mathcal{G}}_i=\sum_{k=0}^i{i \choose k} \mathcal{G}_k & =
& \sum_{k=0}^i{i \choose k}
[C_1t_1^k+C_2t_2^k] \ \ \ \  \ \mbox{(by Eq. (\ref{eee-qqq2}))} \\[0.2cm]
& = & C_1(t_1+1)^i+C_2(t_2+1)^i.
\end{array}
$$
Therefore, we verify Eq. (\ref{e1-new-2012}) by a direct
calculation. Clearly $\check{\mathcal{G}}_0=\mathcal{G}_0=a$ and
$\check{\mathcal{G}}_1=\mathcal{G}_0+\mathcal{G}_1=a+b$. Let us
for the moment assume that $i\geqslant 2$. Then, we have
$${\small
\begin{array}{l}
(r+2)\check{\mathcal{G}}_{i-1}-r\check{\mathcal{G}}_{i-2} \\[0.2cm]
=(r+2)[C_1(t_1+1)^{i-1}+C_2(t_2+1)^{i-1}]-r[C_1(t_1+1)^{i-2}+C_2(t_2+1)^{i-2}]
\\[0.2cm]
=C_1(t_1+1)^{i-2}[(r+2)(t_1+1)-r]+C_2(t_2+1)^{i-2}[(r+2)(t_2+1)-r] \\[0.2cm]
 =C_1(t_1+1)^{i-2}[rt_1+2t_1+2]+C_2(t_2+1)^{i-2}[rt_2+2t_2+2] \\[0.2cm]
 =C_1(t_1+1)^{i-2}[(rt_1+1)+2t_1+1]+C_2(t_2+1)^{i-2}[(rt_2+1)+2t_2+1] \\[0.2cm]
 =C_1(t_1+1)^{i}+C_2(t_2+1)^{i}  \ \ \ \mbox{(note that $rt_1+1=t_1^2$ and $rt_2+1=t_2^2$.) }\\[0.2cm]
 =\check{\mathcal{G}}_i,
\end{array}}
$$
as required.

The second part (i.e., the equalities in  Eq.
(\ref{e2-new-2012})) is similar and we leave it to the
reader.\end{proof}
\begin{lemma}\label{lem5} If $n\geqslant 1$, then $\mathcal{G}_n=a\bar{\mathcal{G}}_{n-1}+b\bar{\mathcal{G}}_n$.
\end{lemma}
\begin{proof} We will proceed by induction with respect to $n$. The case $n=1$ is trivial:
$$ a\bar{\mathcal{G}}_{0}+b\bar{\mathcal{G}}_1=a\cdot 0+b\cdot 1=b=\mathcal{G}_1.$$
Assume the identity is true when $1\leqslant n<k$, that is,
\begin{equation}\label{lem5-iden7}
\mathcal{G}_n=a\bar{\mathcal{G}}_{n-1}+b\bar{\mathcal{G}}_n, \ \
n=1, 2, 3, \ldots, k-1.
\end{equation} We now prove that it is also true for
$n=k$. Easy calculations show that
$$\begin{array}{lll} \mathcal{G}_k=
r\mathcal{G}_{k-1}+\mathcal{G}_{k-2}& = &
r(a\bar{\mathcal{G}}_{k-2}+b\bar{\mathcal{G}}_{k-1})+(a\bar{\mathcal{G}}_{k-3}+b\bar{\mathcal{G}}_{k-2})\\[0.2cm]
& &   \mbox{(by the induction hypothesis, Eq. (\ref{lem5-iden7}))}\\[0.2cm]
& = &
a(r\bar{\mathcal{G}}_{k-2}+\bar{\mathcal{G}}_{k-3})+b(r\bar{\mathcal{G}}_{k-1}+\bar{\mathcal{G}}_{k-2})
\\[0.2cm]
& = & a\bar{\mathcal{G}}_{k-1}+b\bar{\mathcal{G}}_{k},
\end{array}
$$
which completes the proof.
\end{proof}
\begin{lemma}\label{lem6} If $m, n\geqslant 0$, then
\begin{equation}\label{lem6-iden8}
\bar{\mathcal{G}}_{n+m}=\bar{\mathcal{G}}_{n}\bar{\mathcal{G}}_{m-1}+\bar{\mathcal{G}}_{n+1}\bar{\mathcal{G}}_{m},
\end{equation}
and
\begin{equation}\label{lem6-iden9}
\bar{\mathcal{G}}_{m+1}\bar{\mathcal{G}}_{m-1}-\bar{\mathcal{G}}_{m}^2=(-1)^m.
\end{equation}
\end{lemma}
\begin{proof} If $m=0$ or $n=0$, then the proof is straightforward. Therefore, we may assume that $m, n\geqslant1$. Put
$$A:=\left[ \begin{array}{ll}
\bar{\mathcal{G}}_2 & \bar{\mathcal{G}}_1\\[0.2cm]
\bar{\mathcal{G}}_1 & \bar{\mathcal{G}}_0  \end{array}
\right]=\left[ \begin{array}{ll}
r & 1\\[0.2cm]
1 & 0  \end{array} \right].$$ We claim that $$ A^m=\left[
\begin{array}{ll}
\bar{\mathcal{G}}_{m+1} & \bar{\mathcal{G}}_m\\[0.2cm]
\bar{\mathcal{G}}_m & \bar{\mathcal{G}}_{m-1}  \end{array}
\right], \ \ \ \ m\geqslant 1.$$  To prove this we will proceed by
induction with respect to $m$. It is clear for $m=1$. Suppose
that the claim is true for $m-1$. Our task is to show that it is
also true for $m$. Again, we verify the result by a direct
calculation:
$$\begin{array}{lll} A^m = A\cdot A^{m-1}& = & \left[ \begin{array}{ll}
r & 1\\[0.2cm]
1 & 0  \end{array} \right]\cdot \left[
\begin{array}{ll}
\bar{\mathcal{G}}_{m} & \bar{\mathcal{G}}_{m-1}\\[0.2cm]
\bar{\mathcal{G}}_{m-1} & \bar{\mathcal{G}}_{m-2}  \end{array}
\right]\\[0.5cm] & &  \mbox{(by the induction hypothesis)}\\[0.2cm]
& = &   \left[
\begin{array}{ll}
r\bar{\mathcal{G}}_{m}+\bar{\mathcal{G}}_{m-1} & r\bar{\mathcal{G}}_{m-1}+\bar{\mathcal{G}}_{m-2}\\[0.2cm]
\bar{\mathcal{G}}_{m} & \bar{\mathcal{G}}_{m-1}  \end{array}
\right]  \\[0.5cm]
& = & \left[
\begin{array}{ll}
\bar{\mathcal{G}}_{m+1} & \bar{\mathcal{G}}_m\\[0.2cm]
\bar{\mathcal{G}}_m & \bar{\mathcal{G}}_{m-1}  \end{array}
\right],
\end{array}
$$
as claimed.

Now, we have $A^{n+m}=A^n\cdot A^m$, that is
$$ \left[
\begin{array}{ll}
\bar{\mathcal{G}}_{n+m+1} & \bar{\mathcal{G}}_{n+m}\\[0.2cm]
\bar{\mathcal{G}}_{n+m} & \bar{\mathcal{G}}_{n+m-1} \end{array}
\right]=\left[
\begin{array}{ll}
\bar{\mathcal{G}}_{n+1} & \bar{\mathcal{G}}_{n}\\[0.2cm]
\bar{\mathcal{G}}_{n} & \bar{\mathcal{G}}_{n-1}  \end{array}
\right]\cdot \left[
\begin{array}{ll}
\bar{\mathcal{G}}_{m+1} & \bar{\mathcal{G}}_{m}\\[0.2cm]
\bar{\mathcal{G}}_{m} & \bar{\mathcal{G}}_{m-1}  \end{array}
\right],$$  or equivalently
$$ \left[
\begin{array}{ll}
\bar{\mathcal{G}}_{n+m+1} & \bar{\mathcal{G}}_{n+m}\\[0.2cm]
\bar{\mathcal{G}}_{n+m} & \bar{\mathcal{G}}_{n+m-1} \end{array}
\right]=\left[
\begin{array}{ll}
\bar{\mathcal{G}}_{n+1}\bar{\mathcal{G}}_{m+1}+\bar{\mathcal{G}}_{n}\bar{\mathcal{G}}_{m} &
\bar{\mathcal{G}}_{n+1}\bar{\mathcal{G}}_{m}+\bar{\mathcal{G}}_{n}\bar{\mathcal{G}}_{m-1}\\[0.2cm]
\bar{\mathcal{G}}_{n}\bar{\mathcal{G}}_{m+1}+\bar{\mathcal{G}}_{n-1}\bar{\mathcal{G}}_{m}
&
\bar{\mathcal{G}}_{n}\bar{\mathcal{G}}_{m}+\bar{\mathcal{G}}_{n-1}\bar{\mathcal{G}}_{m-1}
\end{array} \right].$$
Comparing $(0,1)$-entries on both sides of this equation yields
$$ \bar{\mathcal{G}}_{n+m}=\bar{\mathcal{G}}_{n+1}\bar{\mathcal{G}}_{m}+
\bar{\mathcal{G}}_{n}\bar{\mathcal{G}}_{m-1},$$
which is the Eq. (\ref{lem6-iden8}).

Moreover, since $\det A=-1$, $\det A^m=(\det A)^m=(-1)^m$, or
equivalently
$$\det\left[
\begin{array}{ll}
\bar{\mathcal{G}}_{m+1} & \bar{\mathcal{G}}_{m}\\[0.2cm]
\bar{\mathcal{G}}_{m} & \bar{\mathcal{G}}_{m-1}  \end{array}
\right]=(-1)^m,$$ and this immediately implies the following:
$$\bar{\mathcal{G}}_{m+1}\bar{\mathcal{G}}_{m-1}-\bar{\mathcal{G}}_{m}^2=(-1)^m,
$$  which is the Eq. (\ref{lem6-iden9}). This completes the proof of the lemma. \end{proof}

An immediate consequence of Lemma \ref{lem6} is the following,
which contains a list of well-known identities.

\begin{corollary}\label{coro-lemma-6} If $m, n\geqslant 0$, then the
following identities hold:
\begin{itemize}
\item[$(i)$]
$\mathcal{F}_n\mathcal{F}_{m-1}+\mathcal{F}_{n+1}\mathcal{F}_m=\mathcal{F}_{m+n}$,
\item[$(ii)$]
$\mathcal{F}_{m+1}\mathcal{F}_{m-1}-\mathcal{F}_{m}^2=(-1)^m$, \
\ \ \mbox{(Cassini's identity)}
\item[$(iii)$]
$\mathcal{P}_n\mathcal{P}_{m-1}+\mathcal{P}_{n+1}\mathcal{P}_m=\mathcal{P}_{m+n}$,
\item[$(iv)$]
$\mathcal{P}_{m+1}\mathcal{P}_{m-1}-\mathcal{P}_{m}^2=(-1)^m$. \
\ \ \mbox{(Simpson's identity)}
\end{itemize}
\end{corollary}

\section{Main Results}
As we mentioned in the Introduction, we are going to obtain some
new matrices whose leading principal minors form the Fibonacci,
Lucas, Pell and Jacobsthal sequences. We start with the following
theorem.

\begin{theorem}\label{th0}
If $\mathcal{G}=\mathcal{G}^{(a,b,r,1)}$, where $a, b$ and $r$ are
integers with $r\geqslant 1$, then for any nonnegative integer $n$
there holds
$$\det T_{\tilde{\mathcal{G}}, \mathcal{G}}(n) =\left\{\begin{array}{lll}
 a & & n=0,\\[0.2cm]
(2b-ar)^{n-1}(a\mathcal{G}_{n-1}+b\mathcal{G}_{n}) & & n\geqslant
1.\end{array} \right.
$$
\end{theorem}
\begin{proof} For $n=0$, we have nothing to do. Therefore, we may assume that $n\geqslant 1$.
We use the  matrix factorization method. Let $T(n)$ denote the
matrix $T_{\tilde{\mathcal{G}}, \mathcal{G}}(n)$. We claim that
\begin{equation}\label{e-new-1} T(n)=L(n)\cdot H(n),\end{equation}
where $L(n)=(L_{i, j})_{0\leqslant i, j\leqslant n}$ with
\begin{equation}\label{e-new-2}
L_{i,j}=\left\{\begin{array}{lll} (-1)^{i+j}\bar{\mathcal{G}}_{i+j-1} &  & j=0,1, \   i\geqslant 0, \\[0.2cm]
0 &  & i=0, \  j\geqslant 2, \\[0.2cm]
L_{i-1, j-1} & & i\geqslant 1, \ j\geqslant 2,
\end{array} \right.\end{equation}
(we recall that $\bar{\mathcal{G}}_{-1}=1$) and
$H(n)=(H_{i,j})_{0\leqslant i, j\leqslant n}$ with
\begin{equation}\label{e-new-3} H_{i,j}=\left\{\begin{array}{lll}
(-1)^{i}\mathcal{G}_i &  & i=0, 1, \ j=0, \\[0.2cm]
\mathcal{G}_{j-i} &  & i=0, 1,  \ j\geqslant 1, \\[0.2cm]
r\mathcal{G}_0-2\mathcal{G}_1 &  & i-j=1, \ i\geqslant 2,  \\[0.2cm]
0 & & \mbox{otherwise}.
\end{array} \right. \end{equation}
The matrix $L(n)$ is a lower triangular matrix with 1's on the
diagonal, whereas $H(n)$ is an upper Hessenberg matrix. The
matrices $L(4)$ and $H(4)$ for instance are given by
$$\begin{array}{lll} L(4)&=&\left (\begin{array}{ccccc}
\bar{\mathcal{G}}_{-1}    & -\bar{\mathcal{G}}_0  & 0 & 0 & 0  \\[0.2cm]
 -\bar{\mathcal{G}}_0  & \bar{\mathcal{G}}_1    & -\bar{\mathcal{G}}_0 & 0 & 0 \\[0.2cm]
 \bar{\mathcal{G}}_{1} &  -\bar{\mathcal{G}}_2  & \bar{\mathcal{G}}_1  & -\bar{\mathcal{G}}_0 & 0 \\[0.2cm]
 -\bar{\mathcal{G}}_2 &  \bar{\mathcal{G}}_3   &  -\bar{\mathcal{G}}_2  & \bar{\mathcal{G}}_1  &
  -\bar{\mathcal{G}}_0 \\[0.2cm]
\bar{\mathcal{G}}_3 & -\bar{\mathcal{G}}_4 & \bar{\mathcal{G}}_3   &  -\bar{\mathcal{G}}_2  & \bar{\mathcal{G}}_1  \\[0.2cm]
\end{array} \right )\\[1.5cm] &=& \left (\begin{array}{ccccc}
1    & 0  & 0 & 0 & 0  \\[0.2cm]
 -\bar{\mathcal{G}}_0  & 1    & 0 & 0 & 0 \\[0.2cm]
 \bar{\mathcal{G}}_{1} &  -\bar{\mathcal{G}}_2  & 1  & 0 & 0 \\[0.2cm]
 -\bar{\mathcal{G}}_2 &  \bar{\mathcal{G}}_3   &  -\bar{\mathcal{G}}_2  & 1  & 0 \\[0.2cm]
\bar{\mathcal{G}}_3 & -\bar{\mathcal{G}}_4 & \bar{\mathcal{G}}_3   &  -\bar{\mathcal{G}}_2  & 1  \\[0.2cm]
\end{array} \right ),\end{array}$$

$$\begin{array}{lll} H(4) &=& \left (\begin{array}{ccccc}
\mathcal{G}_0    & \mathcal{G}_1  & \mathcal{G}_2 & \mathcal{G}_3 & \mathcal{G}_4  \\[0.2cm]
 -\mathcal{G}_1  & \mathcal{G}_0    & \mathcal{G}_1  & \mathcal{G}_2 & \mathcal{G}_3 \\[0.2cm]
 0  &  r\mathcal{G}_0-2\mathcal{G}_1  & 0    & 0 & 0 \\[0.2cm]
0& 0  &  r\mathcal{G}_0-2\mathcal{G}_1  &  0 & 0 \\[0.2cm]
0& 0& 0  &  r\mathcal{G}_0-2\mathcal{G}_1  & 0 \\[0.2cm]
\end{array} \right )\\[1.5cm] &  = & \left (\begin{array}{ccccc}
a    & b  & \mathcal{G}_2 & \mathcal{G}_3 & \mathcal{G}_4  \\[0.2cm]
 -b  & a    & \mathcal{G}_1  & \mathcal{G}_2 & \mathcal{G}_3 \\[0.2cm]
 0  &  ra-2b  & 0    & 0 & 0 \\[0.2cm]
0& 0  &  ra-2b  &  0 & 0 \\[0.2cm]
0& 0& 0  &  ra-2b & 0 \\[0.2cm]
\end{array} \right ).\end{array}$$

In what follows, for convenience, we will let $T=T(n)$, $L=L(n)$
and $H=H(n)$. For the proof of the claimed factorization Eq.
(\ref{e-new-1}) we compute the $(i,j)$-entry of $L\cdot H$, that
is $$(L\cdot H)_{i,j}=\sum\limits_{k=0}^{n}L_{i,k}H_{k,j}.$$
As a matter of fact, we should establish\\
$$\begin{array}{rcl}{\rm R}_0(L\cdot H)&={\rm R}_0(T)=&(\mathcal{G}_0,
\mathcal{G}_1, \ldots, \mathcal{G}_n),\\[0.2cm]
 {\rm C}_0(L\cdot H)&={\rm
C}_0(T)=&(\mathcal{G}_0, -\mathcal{G}_1, \ldots,
(-1)^n\mathcal{G}_n),\end{array}$$ and
\begin{equation}\label{2009-e4}
(L\cdot H)_{i,j}=(L\cdot H)_{i-1,j-1} \ \  \mbox{for} \ \
1\leqslant i, j\leqslant n.
\end{equation}

Let us do the required calculations. First, suppose that $i=0$.
Then
$$(L\cdot H)_{0,j}=\sum_{k=0}^{n}
L_{0,k}H_{k,j}=L_{0,0}H_{0,j}=H_{0,j}=\mathcal{G}_j,$$ which
shows that ${\rm R}_0(L\cdot H)={\rm R}_0(T)$.

Next, assume that $j=0$. In this case, using Lemma \ref{lem5}, we
obtain
$$\begin{array}{lll} (L\cdot H)_{i,0}& = & \sum_{k=0}^{n}
L_{i,k}H_{k,0}=L_{i,0}H_{0,0}+L_{i,1}H_{1,0}\\[0.2cm] & = &
(-1)^i[a\bar{\mathcal{G}}_{i-1}+b\bar{\mathcal{G}}_{i}]=(-1)^i\mathcal{G}_i,\end{array}
$$ and hence we have ${\rm
C}_0(L\cdot H)={\rm C}_0(T)$.

Finally, we must establish Eq. (\ref{2009-e4}). To do this, we
assume that $1\leqslant i, j\leqslant n$. We distinguish two cases
separately: $j=1$ and $j\geqslant 2$.

{\sc Case 1.} $j=1$. We calculate the sum in question:
$$\begin{array}{lll}
(L\cdot H)_{i,1} & = & \sum\limits_{k=0}^{n}
L_{i,k}H_{k,1}=L_{i,0}H_{0,1}+L_{i,1}H_{1,1}+L_{i,2}H_{2,1}\\[0.4cm]
& = &  L_{i,0}H_{0,1}+L_{i,1}H_{1,1}+L_{i-1,1}H_{2,1}\\[0.4cm]
& = &  (-1)^i\bar{\mathcal{G}}_{i-1}b+(-1)^{i+1}\bar{\mathcal{G}}_{i}a+(-1)^i\bar{\mathcal{G}}_{i-1}(ra-2b)\\[0.4cm]
& = & (-1)^i[r\bar{\mathcal{G}}_{i-1}-\bar{\mathcal{G}}_{i}]a+(-1)^{i+1}\bar{\mathcal{G}}_{i-1}b \\[0.4cm]
& = & (-1)^{i+1}\bar{\mathcal{G}}_{i-2}a+(-1)^{i+1}\bar{\mathcal{G}}_{i-1}b \\[0.4cm]
& = & (-1)^{i+1}[a\bar{\mathcal{G}}_{i-2}+b\bar{\mathcal{G}}_{i-1}] \\[0.4cm]
& = & (-1)^{i+1}\mathcal{G}_{i-1} \ \ \ \ (\mbox{by Lemma
\ref{lem5}})\\[0.4cm]
& = & (L\cdot H)_{i-1,0},\\[0.4cm]
\end{array}$$
as required.

{\sc Case 2.} $j\geqslant 2$. We proceed analogously. In this
case, we have
$$\begin{array}{lll}
&& (L\cdot H)_{i,j}\\[0.2cm]
  & = & \sum\limits_{k=0}^{n}
L_{i,k}H_{k,j}=L_{i,0}H_{0,j}+L_{i,1}H_{1,j}+L_{i,2}H_{2,j}+\sum\limits_{k=3}^{n}
L_{i,k}H_{k,j}\\[0.4cm]
& = &
(-1)^i\bar{\mathcal{G}}_{i-1}\mathcal{G}_{j}+(-1)^{i+1}\bar{\mathcal{G}}_{i}\mathcal{G}_{j-1}+\sum\limits_{k=3}^{n}
L_{i-1,k-1}H_{k-1,j-1}\\[0.4cm]
&  &
(\mbox{Note that} \  H_{2,j}=0 \ \mbox{for} \  j\geqslant 2) \\[0.4cm]
& = &
(-1)^i\bar{\mathcal{G}}_{i-1}\mathcal{G}_{j}+(-1)^{i+1}[\bar{\mathcal{G}}_{i-2}+r\bar{\mathcal{G}}_{i-1}]\mathcal{G}_{j-1}
+\sum\limits_{k=2}^{n}
L_{i-1,k}H_{k,j-1}\\[0.4cm]
&  &
(\mbox{Note that} \  L_{i-1,n}=0) \\[0.4cm]
& = &
(-1)^i\bar{\mathcal{G}}_{i-1}[\mathcal{G}_{j}-r\mathcal{G}_{j-1}]+(-1)^{i+1}\bar{\mathcal{G}}_{i-2}\mathcal{G}_{j-1}
+\sum\limits_{k=2}^{n}
L_{i-1,k}H_{k,j-1} \\[0.4cm]
& = &
(-1)^i\bar{\mathcal{G}}_{i-1}\mathcal{G}_{j-2}+(-1)^{i-1}\bar{\mathcal{G}}_{i-2}\mathcal{G}_{j-1}
+\sum\limits_{k=2}^{n}
L_{i-1,k}H_{k,j-1} \\[0.4cm]
& = & \sum\limits_{k=0}^{n} L_{i-1,k}H_{k,j-1}=(L\cdot H)_{i-1,j-1},\\[0.4cm]
\end{array}$$
as required.

Returning back to the Eq. (\ref{e-new-1}), we conclude that
$$\det T= \det(L\cdot H)=\det L\cdot \det H=\det H.$$
Now, expanding the determinant of $H$ along the rows $2, 3,
\ldots, n$, one can easily get
$$\det H=(2b-ra)^{n-1}(a\mathcal{G}_{n-1}+b\mathcal{G}_{n}).$$
This completes the proof of the theorem. \end{proof}

Although we have restricted ourselves to integral values in
Theorem \ref{th0}, it holds for arbitrary values of all
parameters $a$, $b$ and $r$. In the sequel, we will consider some
applications of Theorem \ref{th0} in some simple, but
interesting, cases.

\begin{corollary}\label{corollary-1}
Let $m\geqslant 0$ be an integer and
$\alpha=(\mathcal{F}_{m+i})_{i\geqslant 0}$. Then there holds
\begin{equation}\label{equqtion-17-new} \det T_{\tilde{\alpha},
\alpha}(n)=\mathcal{L}_m^{n-1}\mathcal{F}_{2m+n}=
(\mathcal{F}_{m-1}+\mathcal{F}_{m+1})^{n-1}\mathcal{F}_{2m+n}, \
\ n\geqslant 0.\end{equation} In particular, we have
\begin{itemize}
\item[{$(1)$}] If $m=0$, then  $\det T_{\tilde{\mathcal{F}},
\mathcal{F}}(n)=2^{n-1}\mathcal{F}_n$.
\item[{$(2)$}]  If $m=1$, then $\det T_{\tilde{\alpha},
\alpha}(n)=\mathcal{F}_{n+2}$.
\end{itemize}
\end{corollary}
\begin{proof} Let $d_n=\det T_{\tilde{\alpha},
\alpha}(n)$. Using Theorem \ref{th0} and Corollary
\ref{coro-lemma-6} $(i)$, we have
$$\begin{array}{lll} d_n& =& \left\{\begin{array}{ll}
 \mathcal{F}_m & n=0,\\[0.3cm]
(2\mathcal{F}_{m+1}-\mathcal{F}_m)^{n-1}(\mathcal{F}_m\mathcal{F}_{m+n-1}+\mathcal{F}_{m+1}\mathcal{F}_{m+n})
&  n\geqslant 1,\end{array} \right. \\[0.7cm]
& = & \left\{\begin{array}{lll}
 \mathcal{F}_m & & n=0,\\[0.3cm]
(\mathcal{F}_{m-1}+\mathcal{F}_{m+1})^{n-1}\mathcal{F}_{2m+n} & & n\geqslant 1,\end{array} \right. \\[0.7cm]
& = & \left\{\begin{array}{lll}
 (\mathcal{F}_{m-1}+\mathcal{F}_{m+1})^{-1}\mathcal{F}_{2m} & & n=0,\\[0.3cm]
(\mathcal{F}_{m-1}+\mathcal{F}_{m+1})^{n-1}\mathcal{F}_{2m+n} & & n\geqslant 1,\end{array} \right. \\[0.7cm]
& = & (\mathcal{F}_{m-1}+\mathcal{F}_{m+1})^{n-1}\mathcal{F}_{2m+n}, \ \ \   n\geqslant 0, \\[0.2cm]
& = & \mathcal{L}_{m}^{n-1}\mathcal{F}_{2m+n}, \ \ \   n\geqslant 0. \\[0.2cm]
\end{array}
$$
The rest follows immediately.
\end{proof}

\begin{corollary}\label{corollary-2}
Let $m\geqslant 0$ be an integer and
$\alpha=(\mathcal{P}_{m+i})_{i\geqslant 0}$. Then for every
$n\geqslant 0$ there holds
$$ \det T_{\tilde{\alpha}, \alpha}(n)=(\mathcal{P}_{m-1}+\mathcal{P}_{m+1})^{n-1}\mathcal{P}_{2m+n}. \
$$ In particular, for $m=0$ we conclude that  $\det T_{\tilde{\mathcal{P}},
\mathcal{P}}(n)=2^{n-1}\mathcal{P}_n$.
\end{corollary}
\begin{proof} Let $d_n=\det T_{\tilde{\alpha}, \alpha}(n)$.  Again, using Theorem \ref{th0} and Corollary \ref{coro-lemma-6} $(iii)$, we have
$$\begin{array}{lll} d_n& =& \left\{\begin{array}{lll}
 \mathcal{P}_m & & n=0,\\[0.3cm]
(2\mathcal{P}_{m+1}-2\mathcal{P}_m)^{n-1}(\mathcal{P}_m\mathcal{P}_{m+n-1}+\mathcal{P}_{m+1}\mathcal{P}_{m+n})
& & n\geqslant 1,\end{array} \right. \\[0.7cm]
& = & \left\{\begin{array}{lll}
 \mathcal{P}_m & & n=0,\\[0.3cm]
(\mathcal{P}_{m-1}+\mathcal{P}_{m+1})^{n-1}\mathcal{P}_{2m+n} & & n\geqslant 1,\end{array} \right. \\[0.7cm]
& = & \left\{\begin{array}{lll}
 (\mathcal{P}_{m-1}+\mathcal{P}_{m+1})^{-1}\mathcal{P}_{2m} & & n=0,\\[0.3cm]
(\mathcal{P}_{m-1}+\mathcal{P}_{m+1})^{n-1}\mathcal{P}_{2m+n} & & n\geqslant 1,\end{array} \right. \\[0.7cm]
& = & (\mathcal{P}_{m-1}+\mathcal{P}_{m+1})^{n-1}\mathcal{P}_{2m+n}, \ \ \   n\geqslant 0. \\[0.2cm]
\end{array}
$$
The rest follows immediately.
\end{proof}

Notice that, using Lemma \ref{lem4}, if
$\alpha=(\mathcal{F}_{i+1})_{i\geqslant 0}=(1, 1, 2, 3, 5, 8,
\ldots)$, then
$$\left \{
\begin{array}{lll} \check{\alpha}=(\mathcal{F}_{2i+1})_{i\geqslant 0}=(1, 2,
5, 13, 34, \ldots) &  & \mbox{(A001519 in \cite{integer})} \\[0.3cm]
\check{\tilde{\alpha}}=(\mathcal{F}_{i-1})_{i\geqslant
0}=(1, 0, 1, 1, 2, 3, 5, \ldots) &  &  \mbox{(A212804 in \cite{integer})}   \\
\end{array} \right.
$$

\begin{corollary}\label{coro-seven}
If $\alpha=(\mathcal{F}_{i+1})_{i\geqslant 0}$, then the
following infinite integer matrices:\\

\noindent $\begin{array}{ll} T_{\tilde{\alpha}, \alpha}=\left (
\begin{array}{ccccc}
1& 1& 2& 3 &   \cdot \\
-1& 1& 1& 2 &   \cdot \\
2& -1& 1& 1 &   \cdot \\
-3& 2& -1& 1 &   \cdot \\
\cdot& \cdot& \cdot& \cdot& \cdot \\
\end{array} \right ), &   T_{\alpha, \tilde{\alpha}}=\left (
\begin{array}{ccccc}
1& -1& 2& -3 &   \cdot \\
1& 1& -1& 2 &   \cdot \\
2& 1& 1& -1 &   \cdot \\
3& 2& 1& 1 &   \cdot \\
\cdot& \cdot& \cdot& \cdot& \cdot \\
\end{array} \right ),\\[1.5cm]
P_{\check{\tilde{\alpha}}, \check{\alpha}}=\left (
\begin{array}{ccccc}
1& 2& 5& 13 &  \cdot \\
0& 2& 7& 20 &   \cdot \\
1& 3& 10& 30 &  \cdot \\
1& 4& 14& 44 &  \cdot \\
\cdot& \cdot&  \cdot& \cdot& \cdot \\
\end{array} \right ), &
P_{\check{\alpha}, \check{\tilde{\alpha}}}=\left (
\begin{array}{ccccc}
1& 0& 1& 1 &   \cdot \\
2& 2& 3& 4 &   \cdot \\
5& 7& 10& 14 &  \cdot \\
13& 20& 30 & 44 &  \cdot \\
\cdot& \cdot&  \cdot& \cdot& \cdot \\
\end{array} \right ),\\[1.5cm]
A_{\check{\tilde{\alpha}}, \alpha}=\left (
\begin{array}{ccccc}
1& 1& 2& 3 &   \cdot \\
0& 2& 3& 5 &   \cdot \\
1& 2& 5& 8 &  \cdot \\
1& 3& 7 & 13 &   \cdot \\
\cdot& \cdot&  \cdot& \cdot& \cdot \\
\end{array} \right ), & A_{\check{\alpha}, \tilde{\alpha}}=\left (
\begin{array}{cccccc}
1& -1& 2& -3 &   \cdot \\
2& 0& 1& -1 &   \cdot \\
5& 2& 1& 0 &   \cdot \\
13& 7& 3 & 1 &   \cdot \\
\cdot& \cdot&  \cdot& \cdot& \cdot \\
\end{array} \right ),\\[1.5cm]
\end{array}$

are equimodular. In precisely, we have $$\begin{array}{lll} \det
T_{\tilde{\alpha}, \alpha}(n)&=&\det T_{\alpha,
\tilde{\alpha}}(n)=\det P_{\check{\tilde{\alpha}},
\check{\alpha}}(n)=\det P_{\check{\alpha},
\check{\tilde{\alpha}}}(n)\\[0.2cm] & = & \det A_{\check{\tilde{\alpha}},
\alpha}(n)=\det A_{\check{\alpha},
\tilde{\alpha}}(n)=\mathcal{F}_{n+2},\end{array}$$ for all
$n\geqslant 0$.
\end{corollary}
\begin{proof}
Follows directly from Corollary \ref{corollary-1} $(2)$ and Lemmas
\ref{cor-1} and \ref{lem4}.
\end{proof}
\begin{corollary}\label{th1}
Let $a, b$ and $c$ be arbitrary complex numbers, and for each
positive integer $n$, let $T_n=T_n(a, b, c)$. Then $\det T_1=c$,
$\det T_2=c+1$ and for $n\geqslant 3$,
\begin{equation}\label{e1-new}
\det T_n=\det T_{n-1}+\det T_{n-2},
\end{equation}
if and only if
\begin{equation}\label{e2-new}(a, b, c)\in \Big\{\big(c-\phi,
c-\Phi, c\big), \big(c-\Phi, c-\phi, c\big)\Big\}.
\end{equation}
Furthermore, if Eq. $(\ref{e2-new})$ holds, then we have $\det
T_n=c\mathcal{F}_n+\mathcal{F}_{n-1}$. In particular, there hold

$(1)$ If $c=0$, then $\det T_n=\mathcal{F}_{n-1}$,

$(2)$ If $c=1$, then $\det T_n=\mathcal{F}_{n+1}$,

$(3)$ If $c=-1$, then $\det T_n=-\mathcal{F}_{n-2}$,

$(4)$ If $c=2$, then $\det T_n=\mathcal{F}_{n+2}$,

$(5)$ If $c=-2$, then $\det T_n=-\mathcal{L}_{n-1}$,

$(6)$ If $c=3$, then $\det T_n=\mathcal{L}_{n+1}$.
\end{corollary}
\begin{proof} The sufficiency is clear. To prove the necessity, using Lemma
\ref{lem1} we have
$$\det T_n=(2c-a-b)\det T_{n-1}-(c-a)(c-b)\det T_{n-2},$$
and by comparing this with Eq. (\ref{e1-new}), we then obtain the
system of equalities
$$\left\{\begin{array}{l} 2c-a-b=1 \\[0.2cm]
(c-a)(c-b)=-1  \end{array} \right.$$ which has two solutions of
the forms
$$a=\frac{(2c-1)\pm \sqrt{5}}{2} \ \ \ \ \mbox{and} \ \ \ \
b=\frac{(2c-1)\mp\sqrt{5}}{2}.$$ The rest of the proof is obvious.
\end{proof}

\begin{corollary}\label{th2}
Let $a, b$ and $c$ be arbitrary complex numbers, and for each
positive integer $n$, let $T_n=T_n(a, b, c)$. Then $\det T_1=c$,
$\det T_2=2c+1$ and for $n\geqslant 3$,
\begin{equation}\label{e3-new}
\det T_n=2\det T_{n-1}+\det T_{n-2},
\end{equation}
if and only if
\begin{equation}\label{e4-new}(a, b, c)\in \Big\{\big(c-1+\sqrt{2},
c-1-\sqrt{2}, c\big), \big(c-1-\sqrt{2}, c-1+\sqrt{2},
c\big)\Big\}.
\end{equation}
Furthermore, if Eq. $(\ref{e4-new})$ holds, then we have $\det
T_n=c\mathcal{P}_n+\mathcal{P}_{n-1}$. In particular, there hold

$(1)$ If $c=0$, then $\det T_n= \mathcal{P}_{n-1}$,

$(2)$ If $c=2$, then $\det T_n= \mathcal{P}_{n+1}$.
\end{corollary}
\begin{proof} The sufficiency is clear. To prove the necessity, using Lemma
\ref{lem1} we have
$$\det T_n=(2c-a-b)\det T_{n-1}-(c-a)(c-b)\det T_{n-2},$$
and by comparing this with Eq. (\ref{e3-new}), we then obtain the
system of equalities
$$\left\{\begin{array}{l} 2c-a-b=2 \\[0.2cm]
(c-a)(c-b)=-1  \end{array} \right.$$ which has two solutions of
the forms
$$a=c-1\pm\sqrt{2} \ \ \ \ \mbox{and} \ \ \ \
b=c-1\mp\sqrt{2}.$$ The rest of the proof is obvious.
\end{proof}

\begin{corollary}\label{th3}
Let $a, b$ and $c$ be arbitrary complex numbers, and for each
positive integer $n$, let $T_n=T_n(a, b, c)$. Then $\det T_1=c$,
$\det T_2=c+2$ and for $n\geqslant 3$,
\begin{equation}\label{e5-new}
\det T_n=\det T_{n-1}+2\det T_{n-2},
\end{equation}
if and only if
\begin{equation}\label{e6-new}(a, b, c)\in \Big\{\big(c+1,
c-2, c\big), \big(c-2, c+1, c\big)\Big\}.
\end{equation}
Furthermore, if Eq. $(\ref{e6-new})$ holds, then we have $\det
T_n=c\mathcal{J}_n+2\mathcal{J}_{n-1}$. In particular, we have

$(1)$ If $c=1$, then $\det T_n= \mathcal{J}_{n+1}$.

$(2)$ If $c=3$, then $\det T_n= \mathcal{J}_{n+2}$.
\end{corollary}
\begin{proof} The sufficiency is clear. To prove the necessity, using Lemma
\ref{lem1} we have
$$\det T_n=(2c-a-b)\det T_{n-1}-(c-a)(c-b)\det T_{n-2},$$
and by comparing this with Eq. (\ref{e5-new}), we then obtain the
system of equalities
$$\left\{\begin{array}{l} 2c-a-b=1 \\[0.2cm]
(c-a)(c-b)=-2  \end{array} \right.$$ which has two solutions of
the forms
$$(a, b)\in \big\{(c+1,
c-2), (c-2, c+1)\big\}.$$ The rest of the proof is obvious.
\end{proof}

\begin{corollary}\label{cor-1-new}
Let $\alpha=(\alpha_i)_{i\geqslant 0}$ and
$\beta=(\beta_i)_{i\geqslant 0}$ be two sequences starting with a
common first term. Then there hold:

$(1)$ If $\alpha_i=2^{i+1}-1$ and $\beta_i=2(1-2^{i-1})$, then
$\det P_{\alpha, \beta}(n)= \mathcal{J}_{n+2}$,

$(2)$ If $\alpha_i=2^{i+2}-1$ and $\beta_i=2(2^{i-1}+1)$, then
$\det P_{\alpha, \beta}(n)=\mathcal{J}_{n+3}$.
\end{corollary}
\begin{proof}
Both statements follow directly from Lemma \ref{cor-1}, Lemma
\ref{lem2} and Corollary \ref{th3}.
\end{proof}


\begin{thebibliography}{99}
\bibitem{bacher} R. Bacher, {\em Determinants of matrices related to
the Pascal triangle}, J. Theorie Nombres Bordeaux, 14(2002),
19-41.

\bibitem{bens} A. T. Benjamin and M. A. Shattuck, {\em Recounting determinants
for a class of Hessenberg matrices}, Integers 7 (2007), A55, 7 pp.

\bibitem{by} P. F. Byrd, {\em Problem {\rm B}-$12$: A Lucas determinant},
Fibonacci Quart., 1(4)(1963), p.78.

\bibitem{cahill1} N. D. Cahill, J. R. D'Errico, D. A. Narayan and J. Y.
Narayan, {\em Fibonacci determinants}, College Math. J.,
33(3)(2002), 221-225.

\bibitem{cahill2} N. D. Cahill, J. R. D'Errico and J. P. Spence,
{\em  Complex factorizations of the Fibonacci and Lucas numbers},
Fibonacci Quart., 41(1)(2003), 13-19.

\bibitem{chrs} G. S. Cheon, S. G. Hwang, S. H. Rim and
S. Z. Song, {\em Matrices determined by a linear recurrence
relation among entries}, Special issue on the Combinatorial
Matrix Theory Conference (Pohang, 2002), Linear Algebra Appl.,
373 (2003), 89-99.

\bibitem{GST} K. Griffin, J. L. Stuart and M. J. Tsatsomeros, {\em
Noncirculant Toeplitz matrices all of whose powers are Toeplitz},
Czechoslovak Mathematical Journal,  58(4) (2008), 1185-1193.

\bibitem{mtaj}
A. R. Moghaddamfar, K. Moghaddamfar and  H. Tajbakhsh, {\em New
families of integer matrices whose leading principal minors form
some well-known sequences}, Electronic Journal of Linear Algebra,
22(2011), 598-619.

\bibitem{mp} A. R. Moghaddamfar and S. M. H.
Pooya, {\em Generalized Pascal triangles and Toeplitz matrices},
Electronic Journal of Linear Algebra, 18(2009), 564-588.

\bibitem{mpss} A. R. Moghaddamfar, S. M. H.
Pooya, S. Navid Salehy and S. Nima Salehy, {\em  Fibonacci and
Lucas sequences as the principal minors of some infinite
matrices}, Journal of Algebra and Its Applications, 8(6)(2009),
869-883.

\bibitem{armt} A. R. Moghaddamfar and  H. Tajbakhsh, {\em
Lucas numbers and determinants}, Integers, 12(1)(2012), 21-51.

\bibitem{St1} G. Strang, Introduction to Linear Algebra, Third Edition,
Wellesley-Cambridge Press, 2003.

\bibitem{St2} G. Strang and K. Borre, Linear Algebra, Geodesy, and
GPS. Wellesley-Cambridge Press, 1997.

\bibitem{integer} The On-Line Encyclopedia of Integer Sequences,
 http://www.oeis.org.

\bibitem{yang} Y. Yang and M. Leonard,
{\em Evaluating determinants of convolution-like matrices via
generating functions},  International Journal of Information and
Systems Sciences, 3(4)(2007), 569-580.

\bibitem{zp} H. Zakraj\v{s}ek and M. Petkov\v{s}ek, {\em
 Pascal-like determinants are recursive},
Adv. Appl. Math., 33(3) (2004), 431-450.
\end{thebibliography}
\end{document}